\documentclass[a4paper,11pt]{amsart}
\usepackage{graphicx}
\usepackage{amsmath,amssymb,xy,array}
\usepackage[T1]{fontenc}
\usepackage{amsfonts}
\usepackage{mathabx}
\usepackage{amsmath}
\usepackage{amssymb}
\usepackage{amsthm}
\usepackage{graphicx}
\usepackage{enumitem}
\usepackage{xcolor}
\usepackage[utf8]{inputenc}
\usepackage{hyperref}
\usepackage{mathrsfs}
\usepackage[toc,page]{appendix}
\usepackage{fancyhdr}
\usepackage{tikz}
\usepackage{tikz-cd}
\usepackage{longtable}
\usepackage{geometry}
\usepackage{textcomp}
\usetikzlibrary{trees}
\usetikzlibrary{arrows}
\usepackage{multirow}
\usepackage{youngtab}
\usepackage{mathdots}

\definecolor{gray0}{gray}{0.40}
\definecolor{gray1}{gray}{0.60}
\definecolor{gray2}{gray}{0.70}

\setlist[itemize]{leftmargin=6mm}

\renewcommand{\P}{\mathbb P}

\DeclareMathOperator{\brk}{\underline{rank}}
\DeclareMathOperator{\rk}{rank}
\DeclareMathOperator{\srk}{srk}

\DeclareMathOperator{\rank}{rank}
\DeclareMathOperator{\skrk}{skrk}

\renewcommand{\sec}{\mathbb{S}ec}

\DeclareMathOperator{\Sym}{Sym}

\renewcommand{\P}{\mathbb{P}}

\newtheorem{thm}{Theorem}[section]

\newtheorem{Conjecture}{Conjecture}

\newtheorem{Lemma}[thm]{Lemma}
\newtheorem{Proposition}[thm]{Proposition}

\newtheorem{Corollary}[thm]{Corollary}

\theoremstyle{definition}

\newtheorem{Definition}[thm]{Definition}
\newtheorem{Remark}[thm]{Remark}

\newtheorem{Example}[thm]{Example}

\hypersetup{pdfpagemode=UseNone}
\hypersetup{pdfstartview=FitH}

\begin{document}

\title{On Comon's and Strassen's conjectures}

\date{\today}

\begin{abstract}
Comon's conjecture on the equality of the rank and the symmetric rank
of a symmetric tensor, and Strassen's conjecture on the additivity of
the rank of tensors are two of the most challenging and guiding
problems in the area of tensor decomposition. We survey the main known
results on these conjectures, and, under suitable bounds on the rank,
we prove them, building on classical techniques used in the case of
symmetric tensors, for mixed tensors. Finally, we improve the bound
for Comon's conjecture given by flattenings by producing new equations
for secant varieties of Veronese and Segre varieties.  
\end{abstract}

\author{Alex Casarotti}
\address{Dipartimento di Matematica e Informatica, Universit\`a di Ferrara, Via Machiavelli 35, 44121 Ferrara, Italy}
\email{csrlxa@unife.it}

\author[Alex Massarenti]{Alex Massarenti}
\address{Universidade Federal Fluminense, Rua Alexandre Moura 8 - S\~ao Domingos, 24210-200 Niter\'oi, Rio de Janeiro, Brazil}
\email{alexmassarenti@id.uff.br}

\author{Massimiliano Mella}
\address{Dipartimento di Matematica e Informatica, Universit\`a di Ferrara, Via Machiavelli 35, 44121 Ferrara, Italy}
\email{mll@unife.it}
\maketitle
\setcounter{tocdepth}{1}

\tableofcontents

\section*{Introduction}
Let $X\subset\P^N$ be an irreducible and reduced non-degenerate
variety. The \textit{rank} $\rank_X(p)$ with respect to $X$ of a point
$p\in\mathbb{P}^N$ is the minimal integer $h$ such that $p$ lies in
the linear span of $h$ distinct points of $X$. In particular, if
$Y\subseteq X$ we have that $\rank_X(p)\leq \rank_Y(p)$. 

Since the \textit{$h$-secant variety} $\sec_h(X)$ of $X$ is the subvariety of $\P^N$ obtained as the closure of the union of all $(h-1)$-planes spanned by $h$ general points of $X$, for a general point $p\in \sec_h(X) $ we have $\rank_X(p) = h$.

When the ambient projective space is a space parametrizing tensors we enter the area of tensor decomposition.  A tensor rank decomposition expresses a tensor as a linear combination of simpler tensors. More precisely, given a tensor $T$, lying in a given tensor space over a field $k$, a tensor rank-1 decomposition of $T$ is an expression of the form 
\stepcounter{thm}
\begin{equation}\label{eq1gen}
T = \lambda_1 U_1+...+\lambda_h U_{h}
\end{equation}
where the $U_i$'s are linearly independent rank one tensors, and $\lambda_i\in k^*$. The rank of $T$ is the minimal positive integer $h$ such that $T$ admits such a decomposition. 

Tensor decomposition problems come out naturally in many areas of
mathematics and applied sciences. For instance, in signal processing, numerical linear algebra, computer vision, numerical analysis, neuroscience, graph analysis, control theory and electrical networks \cite{BK09}, \cite{CM96}, \cite{CGLM08}, \cite{LO15}, \cite{MR13}, \cite{MR13C}, \cite{BFFX17}. In pure
mathematics tensor decomposition issues arise while studying the additive decompositions of a general tensor \cite{Do04}, \cite{DK93}, \cite{MM13}, \cite{Ma16},
\cite{RS00}, \cite{TZ11}, \cite{MMS18}.  

Comon's conjecture \cite{CGLM08}, which states the equality of the rank and symmetric rank of a symmetric tensor, and Strassen's conjecture on the additivity of the rank of tensors \cite{St73} are two of the most important and guiding problems in the area of tensor decomposition. 

More precisely, Comon's conjecture predicts that the rank of a homogeneous polynomial $F\in k[x_0,\dots,x_n]_d$ with respect to the Veronese variety $\mathcal{V}_d^n$ is equal to its rank with respect to the Segre variety $\mathcal{S}^{\underline{n}} \cong (\P^n)^d$ into which $\mathcal{V}_d^n$ is diagonally embedded, that is $\rank_{\mathcal{V}_d^n}(F) = \rank_{\mathcal{S}^{\underline{n}}}(F)$.  

Strassen's conjecture was originally stated for triple tensors and then generalized to a number of different contexts. For instance, for homogeneous polynomials it says that if $F\in k[x_0,\dots,x_n]_{d}$ and $G\in k[y_0,\dots,y_m]_{d}$ are homogeneous polynomials in distinct sets of variables then $\rank_{\mathcal{V}_d^{n+m+1}}(F+G) = \rank_{\mathcal{V}_d^n}(F)+\rank_{\mathcal{V}_d^m}(G)$. 

In Sections \ref{SComon} and \ref{SStrassen}, while surveying the state of the art on Comon's and Strassen's conjectures, we push a bit forward some standard techniques, based on catalecticant matrices and more generally on flattenings, to extend some results on these conjectures, known in the setting of Veronese and Segre varieties, for Segre-Veronese and Segre-Grassmann varieties that is to the context of mixed tensors.   

In Section \ref{sec:RNC} we introduce a method to improve a classical result on Comon's conjecture. By standard arguments involving catalecticant matrices it is not hard to prove that Comon's conjecture holds for the general polynomial in $k[x_0,\dots,x_n]_d$ of symmetric rank $h$ as soon as $h < \binom{n+\lfloor\frac{d}{2}\rfloor}{n}$, see Proposition \ref{p3}. We manage to improve this bound looking for equations for the $(h-1)$-secant variety $\sec_{h-1}(\mathcal{V}_d^n)$, not coming from catalecticant matrices, that are restrictions to the space of symmetric tensors of equations of the $(h-1)$-secant variety $\sec_{h-1}(\mathcal{S}^{\underline{n}})$. We will do so by embedding the space of degree $d$ polynomials into the space of degree $d+1$ polynomials by mapping $F$ to $x_0F$ and then considering suitable catalecticant matrices of $x_0F$ rather than those of $F$ itself. 

Implementing this method in Macaulay2 we are able to prove for instance that Comon's conjecture holds for the general cubic polynomial in $n+1$ variables of rank $h = n+1$ as long as $n\leq 30$. Note that for cubics the usual flattenings work for $h\leq n$.
  
\subsection*{Acknowledgments}
The second and third named authors are members of the Gruppo Nazionale per le Strutture Algebriche, Geometriche e le loro Applicazioni of the Istituto Nazionale di Alta Matematica "F. Severi" (GNSAGA-INDAM). We thank the referees for helping us to improve the exposition.

\section{Notation}
Let $\underline{n}=(n_1,\dots,n_p)$ and $\underline{d} = (d_1,\dots,d_p)$ be two $p$-uples of positive integers.
Set 
$$d=d_1+\dots+d_p,\ n=n_1+\dots+n_p,\ {\rm and}\
N(\underline{n},\underline{d})=\prod_{i=1}^p\binom{n_i+d_i}{n_i}$$ 

Let $V_1,\dots, V_p$ be vector spaces of dimensions $n_1+1\leq n_2+1\leq \dots \leq n_p+1$, and consider the product
$$
\mathbb{P}^{\underline{n}} = \mathbb{P}(V_1^{*})\times \dots \times \mathbb{P}(V_p^{*}).
$$
The line bundle 
$$
\mathcal{O}_{\mathbb{P}^{\underline{n}} }(d_1,\dots, d_p)=\mathcal{O}_{\mathbb{P}(V_1^{*})}(d_1)\boxtimes\dots\boxtimes \mathcal{O}_{\mathbb{P}(V_1^{*})}(d_p)
$$
induces an embedding
$$
\begin{array}{cccc}

\sigma\nu_{\underline{d}}^{\underline{n}}:
&\mathbb{P}(V_1^{*})\times \dots \times \mathbb{P}(V_p^{*})& \longrightarrow &
\mathbb{P}(\Sym^{d_1}V_1^{*}\otimes\dots\otimes \Sym^{d_p}V_p^{*})
=\mathbb{P}^{N(\underline{n},\underline{d})-1},\\
      & ([v_1],\dots,[v_p]) & \longmapsto & [v_1^{d_1}\otimes\dots\otimes v_p^{d_p}]
\end{array} 
$$
where $v_i\in V_i$.
We call the image 
$$
\mathcal{SV}_{\underline{d}}^{\underline{n}}= \sigma\nu_{\underline{d}}^{\underline{n}}(\mathbb{P}^{\underline{n}} ) \subset \mathbb{P}^{N(\underline{n},\underline{d})-1}
$$ 
a \textit{Segre-Veronese variety}. It is a smooth  variety of dimension $n$ and degree 
$\frac{(n_1+\dots +n_p)!}{n_1!\dots n_p!}d_1^{n_1}\dots d_p^{n_p}$ in
$\mathbb{P}^{N(\underline{n},\underline{d})-1}$. 

When $p = 1$, $\mathcal{SV}_{d}^{n}$ is a Veronese variety. In this case we write $\mathcal{V}_d^n$ for $\mathcal{SV}_{d}^{n}$, and $\nu_{d}^{n}$ for the Veronese embedding.
When $d_1 = \dots = d_p = 1$, $\mathcal{SV}_{1,\dots,1}^{\underline{n}}$ is a Segre variety. 
In this case we write $\mathcal{S}^{\underline{n}}$ for $\mathcal{SV}_{1,\dots,1}^{\underline{n}}$, and  $\sigma^{\underline{n}}$ for the Segre embedding.
Note that 
$$
\sigma\nu_{\underline{d}}^{\underline{n}}=\sigma^{\underline{n}'}\circ \left(\nu_{d_1}^{n_1}\times \dots \times \nu_{d_p}^{n_p}\right),
$$
where $\underline{n}'=(N(n_1,d_1)-1,\dots,N(n_p,d_p)-1)$.

Similarly, given a $p$-uple of $k$-vector spaces $(V_1^{n_1},...,V_{p}^{n_p})$ and  $p$-uple of positive integers $\underline{d} = (d_1,...,d_p)$ we may consider the Segre-Pl\"ucker embedding 
$$
\begin{array}{cccc}
\sigma\pi_{\underline{d}}^{\underline{n}}:
&Gr(d_1,n_1)\times \dots \times Gr(d_p,n_p)& \longrightarrow &
\mathbb{P}(\bigwedge^{d_1}V_1^{n_1}\otimes\dots\otimes \bigwedge^{d_p}V_p^{n_p})
=\mathbb{P}^{N(\underline{n},\underline{d})-1},\\
      & ([H_1],\dots,[H_p]) & \longmapsto & [H_1\otimes\dots\otimes H_p]
\end{array}
$$ 
where $N(\underline{n},\underline{d}) = \prod_{i=1}^p\binom{n_i}{d_i}$. We call the image 
$$
\mathcal{SG}_{\underline{d}}^{\underline{n}}= \sigma\pi_{\underline{d}}^{\underline{n}}(Gr(d_1,n_1)\times \dots \times Gr(d_p,n_p)) \subset \mathbb{P}^{N(\underline{n},\underline{d})}
$$ 
a \textit{Segre-Grassmann variety}.

\stepcounter{thm}
\subsection{Flattenings}\label{flat}
Let $V_1,...,V_{p}$ be $k$-vector spaces of finite dimension, and consider the tensor product $V_1\otimes ...\otimes V_{p} = (V_{a_1}\otimes ...\otimes V_{a_s})\otimes (V_{b_1}\otimes ...\otimes V_{b_{p-s}})= V_{A}\otimes V_{B}$ with $A\cup B = \{1,...,p\}$, $B = A^c$. Then we may interpret a tensor 
$$T \in V_1\otimes ...\otimes V_p = V_{A}\otimes V_{B}$$
as a linear map $\widetilde{T}:V_{A}^{*}\rightarrow V_{A^c}$. Clearly, if the rank of $T$ is at most $r$ then the rank of $\widetilde{T}$ is at most $r$ as well. Indeed, a decomposition of $T$ as a linear combination of $r$ rank one tensors yields a linear subspace of $V_{A^c}$, generated by the corresponding rank one tensors, containing $\widetilde{T}(V_A^{*})\subseteq V_{A^c}$. The matrix associated to the linear map $\widetilde{T}$ is called an \textit{$(A,B)$-flattening} of $T$.    

In the case of mixed tensors we can consider the embedding
$$\Sym^{d_1}V_1\otimes ...\otimes \Sym^{d_p}V_p\hookrightarrow V_A\otimes V_B$$
where $V_A = \Sym^{a_1}V_1\otimes ...\otimes \Sym^{a_p}V_p$,
$V_B=\Sym^{b_1}V_1\otimes ...\otimes\Sym^{b_p}V_p$, with $d_i =
a_i+b_i$ for any $i = 1,...,p$. In particular, if $n = 1$ we may
interpret a tensor $F\in \Sym^{d_1}V_1$ as a degree $d_1$ homogeneous
polynomial on $\mathbb{P}(V_1^*)$. In this case the matrix associated
to the linear map $\widetilde{F}:V_A^*\rightarrow V_B$ is nothing but
the $a_1$-th \textit{catalecticant matrix} of $F$, that is the matrix
whose rows are the coefficient of the partial derivatives of order
$a_1$ of $F$.

Similarly, by considering the inclusion 
$$\bigwedge^{d_1}V_1\otimes ...\otimes \bigwedge^{d_p}V_p\hookrightarrow V_A\otimes V_B$$
where $V_A = \bigwedge^{a_1}V_1\otimes ...\otimes \bigwedge^{a_p}V_p$, $V_B=\bigwedge^{b_1}V_1\otimes ...\otimes\bigwedge^{b_p}V_p$, with $d_i = a_i+b_i$ for any $i = 1,...,p$, we get the so called \textit{skew-flattenings}. We refer to \cite{La12} for details on the subject.

\begin{Remark}\label{pd} 
The partial derivatives of an homogeneous polynomials are particular flattenings. The partial derivatives of a polynomial $F\in k[x_0,...,x_n]_d$ are $\binom{n+s}{n}$ homogeneous polynomials of degree $d-s$ spanning a linear space $H_{\partial^s F}\subseteq \mathbb{P}(k[x_0,...,x_n]_{d-s})$. 

If $F\in k[x_0,...,x_n]_d$ admits a decomposition as in (\ref{eq1gen}) then $F\in
  \sec_h(\mathcal{V}_{d}^{n})$, and conversely a general $F\in
  \sec_h(\mathcal{V}_{d}^{n})$ can be written as in (\ref{eq1gen}). If $F = \lambda_1L^{d}_{1}+...+\lambda_hL^{d}_{h}$ is a decomposition then the partial derivatives of order $s$ of $F$ can be decomposed as linear combinations of $L^{d-s}_{1},...,L^{d-s}_{h}$ as well. Therefore, the linear space $\left\langle L_1^{d-s},\dots,L_h^{d-s}\right\rangle$ contains $H_{\partial^sF}$.
\end{Remark}

\stepcounter{thm}
\subsection{Rank and border rank}\label{RbR}
Let $X\subset\P^N$ be an irreducible and reduced non-degenerate variety. We define the rank $\rank_X(p)$ with respect to $X$ of a point $p\in\mathbb{P}^N$ as the minimal integer $h$ such that there exist $h$ points in linear general position $x_1,\dots,x_h\in X$ with $p\in\left\langle x_1,\dots,x_h\right\rangle$. Clearly, if $Y\subseteq X$ we have that
\stepcounter{thm}
\begin{equation}\label{ineq}
\rank_X(p)\leq \rank_Y(p)
\end{equation}
The border rank $\brk_X(p)$ of $p\in\mathbb{P}^N$ with respect to $X$ is the smallest integer
$r>0$ such that $p$ is in the Zariski closure of the set of points $q\in\mathbb{P}^N$ such that $\rank_X(q)=r$. In particular $\brk_X(p)\leq \rank_X(p)$. 

Recall that given an irreducible and reduced non-degenerate variety
$X\subset\P^N$, and a positive integer $h\leq N$ the
\textit{$h$-secant variety} $\sec_h(X)$ of $X$ is the subvariety of
$\P^N$ obtained as the Zariski closure of the union of all $(h-1)$-planes
spanned by $h$ general points of $X$. 

In other words $\brk_X(p)$ is computed by the smallest secant variety $\sec_h(X)$ containing $p\in\mathbb{P}^N$.

Now, let $Y,Z$ be subvarieties of an irreducible projective variety $X\subset\mathbb{P}^N$, spanning two linear subspaces $\mathbb{P}^{N_Y} :=\left\langle Y\right\rangle, \mathbb{P}^{N_Z} :=\left\langle Z\right\rangle\subseteq\mathbb{P}^N$. Fix two points $p_Y\in\mathbb{P}^{N_Y}, p_Z\in\mathbb{P}^{N_Z}$, and consider a point $p \in \left\langle p_Y,p_Z\right\rangle$. Clearly
\stepcounter{thm}
\begin{equation}\label{ineqstrass}
\rank_X(p)\leq \rank_Y(p_Y)+\rank_Z(p_Z)
\end{equation}

\section{Comon's conjecture}\label{SComon}
It is natural to ask under which assumptions (\ref{ineq}) is indeed an equality. Consider the Segre-Veronese embedding $\sigma\nu_{\underline{d}}^{\underline{n}}:
\mathbb{P}(V_1^{*})\times \dots \times \mathbb{P}(V_p^{*})\rightarrow
\mathbb{P}(\Sym^{d_1}V_1^{*}\otimes\dots\otimes \Sym^{d_p}V_p^{*})
=\mathbb{P}^{N(\underline{n},\underline{d})-1}$ with $V_1 \cong \dots \cong V_p\cong V$ $k$-vector spaces of dimension $n+1$. Its composition with the diagonal embedding $i:\mathbb{P}(V^{*})\rightarrow\mathbb{P}(V_1^{*})\times \dots \times \mathbb{P}(V_p^{*})$ is the Veronese embedding of $\nu_{d}^{n}$ of degree $d = d_1+\dots +d_p$. Let $\mathcal{V}_{d}^n\subseteq\mathcal{SV}_{\underline{d}}^{\underline{n}}$ be the corresponding Veronese variety. We will denote by $\Pi_{n,d}$ the linear span of $\mathcal{V}_{d}^n$ in $\mathbb{P}^{N(\underline{n},\underline{d})-1}$.

In the notations of Section \ref{RbR} set $X = \mathcal{SV}_{\underline{d}}^{\underline{n}}$ and $Y = \mathcal{V}_{d}^n$.  For any symmetric tensor $T\in \Pi_{n,d}$ we may consider its symmetric rank $\srk(T) := \rank_{\mathcal{V}_{d}^n}(T)$ and its rank $\rk(T):= \rank_{\mathcal{SV}_{\underline{d}}^{\underline{n}}}(T)$ as a mixed tensor. Comon's conjecture predicts that in this particular setting the inequality (\ref{ineq}) is indeed an equality \cite{CGLM08}.

\begin{Conjecture}[Comon's]\label{cconj}
Let $T$ be a symmetric tensor. Then $\rk(T) = \srk(T)$.
\end{Conjecture}
Conjecture \ref{cconj} has been generalized in a number of directions for complex border rank, real rank and real border rank, see \cite[Section 5.7.2]{La12} for a full overview.

Note that when $d = 2$ Comon's conjecture is true. Indeed, $\sec_h(\mathcal{S}^{\underline{n}})$ is cut out by the size $(h+1)\times (h+1)$ minors of a general square matrix and $\sec_h(\mathcal{V}_{2}^n)$ is cut out by the size $(h+1)\times (h+1)$ minors of a general symmetric matrix, that is $\sec_h(\mathcal{V}_{2}^n) = \sec_h(\mathcal{S}^{\underline{n}})\cap\Pi_{n,2}$.

Conjecture \ref{cconj} has been proved in several special cases. For instance, when the symmetric rank is at most two \cite{CGLM08}, when the rank is less than or equal to the order \cite{ZHQ16}, for tensors belonging to tangential varieties to Veronese varieties \cite{BB13}, for tensors in $\mathbb{C}^{2}\otimes\mathbb{C}^{n}\otimes\mathbb{C}^{n}$ \cite{BL13}, when the rank is at most the flattening rank plus one \cite{Fr16}, for the so called Coppersmith–Winograd tensors \cite{LM17}, for symmetric tensors in $\mathbb{C}^4\otimes \mathbb{C}^4\otimes\mathbb{C}^4$ and also for symmetric tensors of symmetric rank at most seven in $\mathbb{C}^n\otimes \mathbb{C}^n\otimes\mathbb{C}^n$ \cite{Se18}.

On the other hand, a counter-example to Comon's conjecture has
recently been found by Y. Shitov \cite{Sh17}. The counter-example
consists of a symmetric tensor $T$ in
$\mathbb{C}^{800}\times\mathbb{C}^{800}\times\mathbb{C}^{800}$ which
can be written as a sum of $903$ rank one tensors but not as a sum of 903
symmetric rank one tensors. It is important to stress that for this
tensor $T$ rank and border rank are quite different. Comon's conjecture for border ranks is still completely open \cite[Problem 25]{Sh17}.

Even though it has been recently proven false in full generality, we believe that Comon's conjecture is true for a general symmetric tensor, perhaps it is even true for those tensor for which $\rank T=\brk T$.

In what follows we use simple arguments based on flattenings to give sufficient conditions for Comon's conjecture, recovering a known result, and its skew-symmetric analogue.

\begin{Lemma}\label{l1}
The tensors $T\in \sec_h(\mathcal{SV}_{\underline{d}}^{\underline{n}})$ such that $\dim (\widetilde{T}(V_A^*))\leq h-1$ for a given flattening $\widetilde{T}$ form a proper closed subset of $\sec_h(\mathcal{SV}_{\underline{d}}^{\underline{n}})$. Furthermore, the same result holds if we replace the Segre-Veronese variety $\mathcal{SV}_{\underline{d}}^{\underline{n}}$ with the Segre-Grassmann variety $\mathcal{SG}_{\underline{d}}^{\underline{n}}$.
\end{Lemma}
\begin{proof}
Let $T\in \sec_h(\mathcal{SV}_{\underline{d}}^{\underline{n}})$ be a general
point. Assume that $\dim (\widetilde{T}(V_A^*))\leq h-1$. This condition forces the $(A,B)$-flattening matrix to have rank
at most $h-1$. On the other hand, by \cite[Proposition 4.1]{SU00} these minors do not vanish on $\sec_h(\mathcal{SV}_{\underline{d}}^{\underline{n}})$, and therefore define a proper closed subset of $\sec_h(\mathcal{SV}_{\underline{d}}^{\underline{n}})$. In the Segre-Grassmann setting we argue in the same way by using skew-flattenings.
\end{proof}

\begin{Proposition}\cite{IK99}\label{p3}
For any integer $h <\binom{n+\lfloor\frac{d}{2}\rfloor}{n}$ there exists an open subset $\mathcal{U}_h\subseteq\sec(\mathcal{V}_n^d)$ such that for any $T\in \mathcal{U}_h$ the rank and the symmetric rank of $T$ coincide, that is
$$\rk(T) = \srk(T)$$
\end{Proposition} 
\begin{proof}
First of all, note that we always have $\rk(T)\leq\srk(T)$. Furthermore, Section \ref{flat} yields that for any $(A,B)$-flattening $\widetilde{T}:V_A^{*}\rightarrow V_B$ the inequality $\rk(T)\geq \dim(\widetilde{T}(V_A^{*}))$ holds. Since $T$ is symmetric and its catalecticant matrices are particular flattenings we get that $\rk(T)\geq \dim(H_{\partial^sT})$ for any $s\geq 0$.

Now, for a general $T\in \sec_h(\mathcal{V}_{d}^n)$ we have $\srk(T) = h$, and if $h <\binom{n+\overline{s}}{n}$, where $\overline{s} = \lfloor\frac{d}{2}\rfloor$, then Lemma \ref{l1} yields $\dim(H_{\partial^{\overline{s}}T}) = h$. Therefore, under these conditions we have the following chain of inequalities
$$\dim(H_{\partial^{\overline{s}}T})\leq \rk(T)\leq \srk(T) = \dim(H_{\partial^{\overline{s}}T})$$
and hence $\rk(T) = \srk(T)$. 
\end{proof}

Now, consider the Segre-Pl\"ucker embedding $\mathbb{P}(V_{1})\times\ldots\times\mathbb{P}(V_{p})\rightarrow\mathbb{P}(\bigwedge^{d_{1}}V_{1}\otimes\dots\otimes\bigwedge^{d_{p}}V_{p})=\mathbb{P}^{N(\underline{n},\underline{d})-1}$ with $V_{1}\cong\ldots\cong V_{p}\cong V$ $k$-vector spaces of dimension $n+1$. Its composition with the diagonal embedding $i:\mathbb{P}(V)\rightarrow\mathbb{P}(V_{1})\times\dots\times\mathbb{P}(V_{p})$ is the Pl\"ucker embedding of $Gr(d,n)$ with $d=d_{1}+\ldots+d_{p}$. Let $Gr(d,n)\subseteq\mathcal{SG}_{\underline{d}}^{\underline{n}}$ be the corresponding Grassmannian and let us denote by $\Pi_{n,d}$ its linear span in $\mathbb{P}^{N(\underline{n},\underline{d})-1}$. 

For any skew-symmetric tensor $T\in\Pi_{n,d}$ we may consider its skew rank $\skrk(T)$ that is its rank with respect to the Grassmannian $Gr(d,n)\subseteq\Pi_{n,d}$, and its rank $\rk(T)$ as a mixed tensor. Playing the same game as in Proposition \ref{p3} we have the following.

\begin{Proposition}\label{p4}
For any integer $h<\binom{n}{\lfloor\frac{d}{2}\rfloor}$ there exists an open subset $\mathcal{U}_{h}\subseteq\sec_{h}(Gr(d,n))$ such that for any $T\in\mathcal{U}_{h}$ the rank and the skew rank of $T$ coincide, that is 
$$\rk(T)=\skrk(T)$$
\end{Proposition}
\begin{proof}
As before for any tensor $T$ we have $\rk(T)\leq \skrk(T)$. For any $(A,B)$-skew-flattening $\widetilde{T}:V_A^{*}\rightarrow V_B$ we have $\skrk(T)\geq \dim(\widetilde{T}(V_A^{*}))$. Furthermore, since $\widetilde{T}$ is in particular a flattening also the inequality $\rk(T)\geq \dim(\widetilde{T}(V_A^{*}))$ holds.
	
Now, for a general $T\in\sec_{h}(Gr(d,n))$ we have $\skrk(T)=h$, and if $h<\tbinom{n}{\overline{s}}$, where $\overline{s}=\lfloor\frac{d}{2}\rfloor$, Lemma \ref{l1} yields $\skrk(T)=\dim(\widetilde{T}_{\overline{s}}(V_{A}^{*}))$, where $\widetilde{T}_{\overline{s}}$ is the skew-flattening corresponding to the partition $(\overline{s},d-\overline{s})$ of $d$. Therefore, we deduce that 
$$\dim(\widetilde{T}_{\overline{s}}(V_{A}^{*}))\leq \rk(T)\leq \skrk(T)=\dim(\widetilde{T}_{\overline{s}}(V_{A}^{*}))$$ 
and hence $\rk(T)=\skrk(T)$. 
\end{proof}

\begin{Remark}\label{embdeg}
Propositions \ref{p3}, \ref{p4} suggest that whenever we are able to write determinantal equations for secant varieties we are able to verify Comon's conjecture. We conclude this section suggesting a possible way to improve the range where the general Comon's conjecture holds giving a conjectural way to produce determinantal equations for some secant varieties.

Set $\underline{n}=(n,\ldots,n)$, $(d+1)$-times, $\underline{n}_1=(n,\ldots, n)$, $d$-times, and consider the
corresponding Segre varieties $X:=\mathcal{S}^{\underline{n}}$, $X_1:=\mathcal{S}^{\underline{n}_1}$ and Veronese varieties $Y=\mathcal{V}^n_{d+1}$, $Y_1:=\mathcal{V}^n_d$. Fix the polynomial $x_0^{d+1}\in Y$ and let $\Pi$ be the linear space spanned by the polynomials of the form $x_0 F$, where $F$ is a polynomial of degree $d$. This allow us to see $Y_1\subseteq \Pi$. Note that polynomials of the form $x_0 L_1^d$ lie in the tangent space of $Y$ at $L_1^{d+1}$, and therefore $\brk_{Y}(x_0L^{\otimes d})=2$.

Hence for a polynomial $F$ of degree $d$ we have $\rank_{Y}(x_0F)\leq 2\rank_{Y_1}(F)$. Our aim is to understand when the equality holds.

We may mimic the same construction for the Segre varieties $X$ and
$X_1$, and use determinantal equations for the secant varieties of
$X_1$ to give determinantal equations of the secant varieties of $X$
and henceforth conclude Comon's conjecture. In particular, as soon as
$d$ is odd and $d<n$, this produces new determinantal equations for
$\sec_h(X_1)$ and $\sec_h(Y_1)$ with
$2h<\binom{n+\frac{d+1}2}{n}$. Therefore, this would give new cases in
which the general Comon's conjecture holds. Unfortunately, we are only
able to successfully implement this procedure in very special cases,
see Section~\ref{sec:RNC}.
\end{Remark}

\section{Strassen's conjecture}\label{SStrassen}
Another natural problem consists in giving hypotheses under which in (\ref{ineqstrass}) equality holds. Consider the triple Segre embedding $\sigma^{\underline{n}}:
\mathbb{P}(V_1^{*})\times \mathbb{P}(V_2^{*})\times\mathbb{P}(V_3^{*}) = \mathbb{P}^a\times \mathbb{P}^b\times\mathbb{P}^c\rightarrow
\mathbb{P}(V_1^{*}\otimes V_2^{*}\otimes V_1^{*})
=\mathbb{P}^{N(\underline{n},\underline{d})-1}$, and let $\mathcal{S}^{\underline{n}}$ be the corresponding Segre variety. Now, take complementary subspaces $\mathbb{P}^{a_1},\mathbb{P}^{a_2}\subset \mathbb{P}^{a}$, $\mathbb{P}^{b_1},\mathbb{P}^{b_2}\subset \mathbb{P}^{b}$, $\mathbb{P}^{c_1},\mathbb{P}^{c_2}\subset \mathbb{P}^{c}$, and let $\mathcal{S}^{(a_1,b_1,c_1)}, \mathcal{S}^{(a_2,b_2,c_2)}$ be the Segre varieties associated respectively to $\mathbb{P}^{a_1}\times\mathbb{P}^{b_1}\times\mathbb{P}^{c_1}$ and $\mathbb{P}^{a_2}\times\mathbb{P}^{b_2}\times\mathbb{P}^{c_2}$. 

In the notations of Section \ref{RbR} set $X = \mathcal{S}^{\underline{n}}$, $Y = \mathcal{S}^{(a_1,b_1,c_1)}$ and $Z=\mathcal{S}^{(a_1,b_1,c_1)}$. Strassen's conjecture states that the additivity of the rank holds for triple tensors, or in onther words that in this setting the inequality (\ref{ineqstrass}) is indeed an equality \cite{St73}. 
  
\begin{Conjecture}[Strassen's]\label{sconj}
In the above notation let $T_1\in \left\langle \mathcal{S}^{(a_1,b_1,c_1)}\right\rangle, T_2\in \left\langle \mathcal{S}^{(a_2,b_2,c_2)}\right\rangle$ be two tensors. Then $\rank(T_1\oplus T_2) = \rank(T_1)+\rank(T_2)$.
\end{Conjecture}

Even though Conjecture \ref{sconj} was originally stated in the context of triple tensors that is bilinear forms, with particular attention to the complexity of matrix multiplication, a number of generalizations are immediate. For instance, we could ask the same question for higher order tensors, symmetric tensors, mixed tensors and skew-symmetric tensors. It is also natural to ask for the analogue of Conjecture \ref{sconj} for border rank. This has been answered negatively \cite{Sc81}.

Conjecture \ref{sconj} and its analogues have been proven when either $T_1$ or $T_2$ has dimension at most two, when $\rank(T_1)$ can be determined by the so called substitution method \cite{LM17}, when $\dim(V_1) = 2$ both for the rank and the border rank \cite{BGL13}, when $T_1,T_2$ are symmetric that is homogeneous polynomials in disjoint sets of variables, either $T_1,T_2$ is a power, or both $T_1$ and $T_2$ have two variables, or either $T_1$ or $T_2$ has small rank \cite{CCC15}, and also for other classes of homogeneous polynomials \cite{CCO17}, \cite{Te15}. 

As for Comon's conjecture a counterexample to Strassen's conjecture has recentely been given by Y. Shitov \cite{Sh17b}. In this case Y. Shitov proved that over any infinite field there exist tensors $T_1,T_2$ such that the inequality in Conjecture \ref{sconj} is strict.

In what follows we give sufficient conditions for Strassen's conjecture, recovering a known result, and for its mixed and skew-symmetric analogues.

\begin{Proposition}\cite{IK99}\label{p1}
Let $V_1,V_2$ be $k$-vector spaces of dimensions $n+1,m+1$, and consider $V = V_1\oplus V_2$. Let $F\in \Sym^d(V_1)\subset\Sym^d(V)$ and $G\in \Sym^d(V_2)\subset\Sym^d(V)$ be two homogeneous polynomials. If there exists an integer $s>0$ such that 
$$\dim(H_{\partial^s F}) = \srk(F),\quad \dim(H_{\partial^s G}) = \srk(G)$$
then $\srk(F+G) = \srk(F)+\srk(G)$.
\end{Proposition}
\begin{proof}
Clearly, $\srk(F+G) \leq \srk(F)+\srk(G)$ holds in general. On the other hand, our hypothesis yields
$$\srk(F)+\srk(G) = \dim(H_{\partial^s F})+\dim(H_{\partial^s F}) = \dim(H_{\partial^s F+G})\leq \srk(F+G)$$
where the last inequality follows from Remark \ref{pd}.
\end{proof}

\begin{Remark}
The argument used in the proof of Proposition \ref{p1} works for $F\in\mathbb{P}^{N(n,d)}$ general only if for the generic rank we have $\lfloor\frac{\tbinom{n+d}{d}}{n+1}\rfloor\leq\binom{n+\lfloor\frac{d}{2}\rfloor}{n}$. For instance, when $n = 3,d=6$ the generic rank is $21$ while the maximal dimension of the spaces spanned by partial derivatives is $20$.
\end{Remark}  

\begin{Proposition}\label{p1b}
Let $V_{1},\ldots,V_{p}$ and $W_{1},\ldots,W_{p}$ be $k$-vector spaces of dimension $n_{1}+1,\ldots,n_{p}+1$ and $m_{1}+1,\ldots,m_{p}+1$ respectively. Consider $U_{i}=V_{i}\oplus W_{i}$ for every $1\leq i\leq p$. Let $T_{1}\in \Sym^{d_{1}}V_{1}\otimes\dots\otimes \Sym^{d_{p}}V_{p}\subset \Sym^{d_{1}}U_{1}\otimes\dots\otimes \Sym^{d_{p}}U_{p}$ and $T_{2}\in \Sym^{d_{1}}W_{1}\otimes\dots\otimes \Sym^{d_{p}}W_{p}\subset \Sym^{d_{1}}U_{1}\otimes\dots\otimes \Sym^{d_{p}}U_{p}$ be two mixed tensors. 

If for any $i\in\{1,..,p\}$ there exists a pair $(a_{i},b_{i})$ with $a_{i}+b_{i}=d_{i}$ and $(A,B)$-flattenings $\widetilde{T}_1:V_A^{*}\rightarrow V_B$, $\widetilde{T}_2:V_A^{*}\rightarrow V_B$ as in (\ref{flat}) such that 
$$\dim(\widetilde{T}_1(V_A^{*}))=\rk(T_{1}), \quad  \dim(\widetilde{T}_2(V_A^{*}))=\rk(T_{2})$$ 
then $\rk(T_{1}+T_{2})=\rk(T_{1})+\rk(T_{2})$.
\end{Proposition}
\begin{proof}
Clearly, $\rk(T_{1}+T_{2})\leq \rk(T_{1})+\rk(T_{2})$. On the other hand, our hypothesis yields 
$$\rk(T_{1})+\rk(T_{2})=\dim(\widetilde{T}_1(V_A^{*}))+\dim(\widetilde{T}_2(V_A^{*}))=\dim(\widetilde{T_{1}+T_{2}}(V_A^{*}))\leq \rk(T_{1}+T_{2})$$
where $\widetilde{T_{1}+T_{2}}$ denotes the $(A,B)$-flattening of the mixed tensor $T_1+T_2$.
\end{proof}

Arguing as in the proof of Proposition \ref{p1b} with skew-symmetric flattenings we have an analogous statement in the Segre-Grassmann setting.

\begin{Proposition}
Let $V_{1},\ldots,V_{p}$ and $W_{1},\ldots,W_{p}$ be $k$-vector spaces of dimension $n_{1}+1,\dots,n_{p}+1$ and $m_{1}+1,\dots,m_{p}+1$ respectively. Consider $U_{i}=V_{i}\oplus W_{i}$ for every $1\leq i\leq p$, and let $T_{1}\in\bigwedge^{d_{1}}V_{1}\otimes\dots\otimes\bigwedge^{d_{p}}V_{p}\subset\bigwedge^{d_{1}}U_{1}\otimes\dots\otimes\bigwedge^{d_{p}}U_{p}$ and $T_{2}\in\bigwedge^{d_{1}}W_{1}\otimes\dots\otimes\bigwedge^{d_{p}}W_{p}\subset\bigwedge^{d_{p}}U_{1}\otimes\dots\otimes\bigwedge^{d_{p}}U_{p}$ be two skew-symmetric tensors with $d_{i}\leq \min\{n_{i}+1,m_{i}+1\}$. 

If for any $i\in\{1,\ldots,p\}$ there exists a pair $(a_{i},b_{i})$ with $a_{i}+b_{i}=d_{i}$ and $(A,B)$-skew-flattenings $\widetilde{T}_1:V_A^{*}\rightarrow V_B$, $\widetilde{T}_2:V_A^{*}\rightarrow V_B$ as in (\ref{flat}) such that 
$$\dim(\widetilde{T}_1(V_A^{*}))=\rk(T_{1}), \quad  \dim(\widetilde{T}_2(V_A^{*}))=\rk(T_{2})$$ 
then $\rk(T_{1}+T_{2})=\rk(T_{1})+\rk(T_{2})$.
\end{Proposition}

\section{On the rank of $x_0F$}\label{sec:RNC}
In this section, building on Remark \ref{embdeg}, we present new cases in which Comon's conjecture holds. Recall, that for a smooth point $x\in X$, the \textit{$a$-osculating space} $\mathbb{T}_x^{a}X$ of $X$ at $x$ is roughly the smaller linear subspace locally approximating $X$ up to order $a$ at $x$, and the \textit{$a$-osculating variety} $T^{a}X$ of $X$ is defined as the closure of the union of all the osculating spaces
$$
T^{a}X = \overline{{\bigcup_{x\in X}}\mathbb{T}^{a}_{x}X}
$$
For any $1\leq a\leq d-1$ the osculating space
$\mathbb{T}^{a}_{[L^{d}]}\mathcal{V}^n_d$ of order $a$ at the point $[L^{d}]\in V_{d}$ can be written as  
$$
\mathbb{T}^{a}_{[L^{d}]}\mathcal{V}^n_d = \left\langle L^{d-a}F \; | \; F\in k[x_0,\ldots,x_n]_{a}\right\rangle\subseteq\mathbb{P}^{N}
$$
Equivalently, $\mathbb{T}^{a}_{[L^{d}]}\mathcal{V}^n_d$ is the space of homogeneous polynomials whose derivatives of order less than or equal to $a$ in the direction given by the linear form $L$ vanish. Note that
$\dim(\mathbb{T}^{a}_{[L^{d}]}\mathcal{V}^n_d) = {{n+a}\choose{n}}-1$ and $\mathbb{T}^{b}_{[L^{d}]}\mathcal{V}^n_d\subseteq\mathbb{T}^{a}_{[L^{d}]}\mathcal{V}^n_d$ for any $b\leq a$. Moreover, for any $1\leq a\leq d$ and $[L^{d}]\in \mathcal{V}^n_d$ we can embed a copy of $\mathcal{V}^n_a$ into the osculating space $\mathbb{T}^{a}_{[L^{d}]}\mathcal{V}^n_d$ by considering 
$$\mathcal{V}^n_a = \{L^{d-a}M^{a}\mid M\in k[x_0,\ldots,x_n]_1\}\subseteq \mathbb{T}^{a}_{[L^{d}]}\mathcal{V}^n_d$$ 

\begin{Remark}\label{remfun}
Let us expand the ideas in Remark \ref{embdeg}. We can embed  
$$\mathcal{V}^n_d = \{x_0 L^{d}\mid L\in k[x_0,\ldots,x_n]_1\}\subseteq \mathbb{T}^{d}_{[x_0^{d}]}\mathcal{V}^n_{d+1}$$ 
and Remark \ref{embdeg} yields that
\stepcounter{thm}
\begin{equation}\label{cont}
\mathbb{S}ec_{h}(\mathcal{V}^n_{d})\subseteq\mathbb{S}ec_{2h}(\mathcal{V}^n_{d+1})\cap
\mathbb{T}^{d}_{[L^{d+1}]}\mathcal{V}^n_{d+1}
\end{equation}
This embedding extends to an embedding at the level of Segre varieties, and, in the notation of Remark \ref{embdeg}, we have that $\mathbb{S}ec_{h}(\mathcal{S}^{n_1})\subseteq\mathbb{S}ec_{2h}(\mathcal{S}^{n})$.

Assume that for a polynomial $F\in \mathbb{S}ec_{h}(\mathcal{V}^n_{d})$ we have $F\in\mathbb{S}ec_{h-1}(\mathcal{S}^{n_1})$. Then $x_0F \in\mathbb{S}ec_{2h-2}(\mathcal{S}^{n})$. Now, if we find a determinantal equation of $\mathbb{S}ec_{2h-2}(\mathcal{V}^n_{d+1})$ coming as the restriction to $\Pi$, the space of symmetric tensors, of a determinantal equation of $\mathbb{S}ec_{2h-2}(\mathcal{S}^{n})$, and not vanishing at $x_0F$ then $x_0F \notin\mathbb{S}ec_{2h-2}(\mathcal{S}^{n})$ and hence $F\notin\mathbb{S}ec_{h-1}(\mathcal{S}^{n_1})$ proving Comon's conjecture for $F$. 

This will be the leading idea to keep in mind in what follows. The determinantal equations involved will always come from minors of suitable catalecticant matrices, that can be therefore seen as the restriction to $\Pi$ of determinantal equations for the secants of the Segre coming from non symmetric flattenings.        
\end{Remark}

It is easy to give examples where the inequality (\ref{cont}) is strict. When $n = 1$ the generic rank is $g_d=\lceil\frac{d+1}{2}\rceil$. Then for $d$ odd we have $g_d = g_{d-1}$ while for $d$ even we have $g_d=g_{d-1}+1$. Hence $\rk_{\mathcal{V}_d}x_0F<2 \rk_{\mathcal{V}_{d-1}}F$ if $2\rk_{\mathcal{V}_{d-1}}F>\frac{g_d}{2}$, where $\mathcal{V}_d:=\mathcal{V}_d^1$ is the rational normal curve. It is natural to ask if the inequality is indeed an equality as long as the rank is subgeneric. In the case $n = 1$ we have the following result.

\begin{Proposition}\label{p5}
Let $\mathcal{V}_d:=\mathcal{V}_d^1$ be the degree $d$ rational normal curve. If $2h<g_{d+1}$ then there does not exist $k_{h}>0$ such that $\mathbb{S}ec_{h}(\mathcal{V}_{d})\subseteq\mathbb{S}ec_{2h-k_h}(\mathcal{V}_{d+1})\cap \mathbb{T}^{d}_{[x^{d+1}]}\mathcal{V}_{d+1}$.
\end{Proposition}
\begin{proof}
Clearly, it is enough to prove the statement for $k_{h}=1$. Let $p\in\sec_{h}(\mathcal{V}_{d})$ be a general point. Then $p\in \left\langle [x_{0}L_{1}^{d}],\dots,[x_{0}L_{h}^{d}]\right\rangle$ with $L_{i}$ general linear forms. In particular 
$$p \in H:=\left\langle\mathbb{T}_{[L_{1}^{d+1}]}\mathcal{V}_{d+1},\dots,\mathbb{T}_{[L_{h}^{d+1}]}\mathcal{V}_{d+1}\right\rangle$$ 
Note that $\dim(H) = 2h-1$. Now, assume that $p$ is contained also in $\sec_{2h-1}(\mathcal{V}_{d+1})$. Then there exists a linear subspace $H'\subset\mathbb{P}^{d+1}$ of dimension $2h-2$ passing through $p$ intersecting $\mathcal{V}_{d+1}$ at $2h-1$ points $q_1,\dots,q_r$ counted with multiplicity. Let $q_{i_1},\dots,q_{i_r}$ be the points among the $q_i$ coinciding with some of the $[L_i^{d+1}]$ and such that the intersection multiplicity of $H'$ and $\mathcal{V}_{d+1}$ at $q_{i_j}$ is one, and $q_{j_1},\dots,q_{j_r}$ be the points among the $q_i$ coinciding with some of the $[L_i^{d+1}]$ and such that the intersection multiplicity of $H'$ and $\mathcal{V}_d$ at $q_{j_k}$ is greater that or equal to two.

Set $\Pi := \left\langle H,H'\right\rangle$, then $\dim(\Pi) = 2h-1+2h-2-i_r-2j_r$ and $\Pi$ intersects $\mathcal{V}_{d+1}$ at $2h+(2h-1-i_r-2j_r)$ points counted with multiplicity. Consider general points $b_1,\dots,b_s \in \mathcal{V}_{d+1}$ with $s = i_r+2j_r$, and the linear space $\Pi' = \left\langle \Pi,b_1,\dots,b_s\right\rangle$. Therefore, $\dim(\Pi') = 4h-3$ and $\Pi'$ intersects $\mathcal{V}_{d+1}$ at $4h-1$ points counted with multiplicity. Since $2h\leq \frac{d+3}{2}$ adding enough general points to $\Pi'$ we may construct a hyperplane in $\mathbb{P}^{d+1}$ intersecting $\mathcal{V}_{d+1}$ at $d+2$ points counted with multiplicity, a contradiction. 
\end{proof}

Proposition~\ref{p5} can be applied to get results on the rank of a special class of matrices called Hankel matrices.

Let $F=Z_{0}x_{0}^{d}+\ldots+\tbinom{d}{d-i}Z_{i}x_{0}^{d-i}x_{1}^{i}+\ldots+Z_{d}x_{1}^{d}$
be a binary form and consider $[Z_{0},\ldots,Z_{d}]$ as homogeneous coordinates on $\mathbb{P}(k[x_0,x_1]_d)$. Furthermore, consider the matrices
$$
M_{2n} = \left(\begin{array}{ccc}
Z_{0} & \dots & Z_{n}\\
\vdots &\ddots & \vdots\\
Z_{n} & \dots & Z_{d}
\end{array}\right),
\quad
M_{2n+1} = \left(\begin{array}{ccccc}
Z_{0} & \dots & Z_{n}\\
\vdots & \ddots & \vdots\\
Z_{n+1} & \dots & Z_{d}
\end{array}\right)
$$
It is well known that the ideal of $\sec_{h}(V_{d})$ is cut out by the
minors of $M_d$ of size $(h+1)\times (h+1)$ \cite{LO15}.

Now, consider a polynomial $F\in k[x_{0},x_{1}]_{d}$ with
homogeneous coordinates $[Z_{0},\ldots,Z_{d}]$. Then $F' := x_{0}F\in k[x_{0},x_{1}]_{d+1}$ has homogeneous coordinates $[Z_{0}',\ldots,Z_{d+1}']$ with 
$$Z_{i}'=\frac{d+1-i}{d+1}Z_{i}$$ 

In order to determine the rank of $F'$ we have to relate the rank of the matrices 
$$
N_{2n} = \left(\begin{array}{cccc}
Z_{0} & \frac{d}{d+1}Z_{1} & \dots & \frac{d+1-n}{d+1}Z_n\\
\frac{d}{d+1}Z_{1} & \dots  & \dots & \frac{d-n}{d+1}Z_{n+1}\\
\vdots & \ddots & \ddots & \vdots\\
\frac{d-n+2}{d+1}Z_{n-1} & \dots  & \dots & \frac{1}{d+1}Z_{d}\\
\frac{d-n+1}{d+1}Z_n & \dots & \frac{1}{d+1}Z_{d} & 0
\end{array}\right)
$$
$$
N_{2n+1} = \left(\begin{array}{cccc}
Z_{0} & \frac{d}{d+1}Z_{1} & \dots & \frac{d-n}{d+1}Z_{n}\\
\frac{d}{d+1}Z_{1} & \dots  & \dots & \frac{d-n-1}{d+1}Z_{n+2}\\
\vdots & \ddots & \ddots & \vdots\\
\frac{d-n+2}{d+1}Z_{n} & \dots  & \dots & \frac{1}{d+1}Z_{d}\\
\frac{d-n+1}{d+1}Z_{n+1} & \dots & \frac{1}{d+1}Z_{d} & 0
\end{array}\right)
$$
with the rank of $M_{d}$. 

\begin{Definition}
A matrix $A =(A_{i,j})\in M(a,b)$ such that $A_{i,j}=A_{h,k}$ whenever $i+j=h+k$ is called a \textit{Hankel matrix}.
\end{Definition}
In particular all the matrices of the form $M_d$ and $N_d$ considered above are Hankel matrices. 

Let $M(a,b)$ be the vector space of $a\times b$ matrices with coefficients in the base field $k$. For any $h\leq \min\{a,b\}$ let $Rank_{r}(M(a,b))\subseteq M(a,b)$ be the subvariety consisting of all matrices of rank at most $h$.
 
Now, consider the map $\beta:\mathbb{N}\longrightarrow\mathbb{N}\times\mathbb{N}$ given by $\beta(2n)=(n+1,n+1)$ and $\beta(2n+1)=(n+2,n+1)$. For any $d\geq 1$ we can view the subspace $H_{d}\subseteq M(\beta(d))$ formed by matrices of the form $M_d$ as the subspace of Hankel matrices. Now, given any linear morphism $f:M(a,b)\rightarrow M(c,d)$ we can ask if for some $s\leq \min\{c,d\}$ we have $f(Rank_{h}(M(a,b)))\subseteq Rank_{s}(M(c,d))$.
 
\begin{Corollary}
	Consider the linear morphism
	$$
	\begin{array}{cccc}
	\alpha_{d}: &M(\beta(d))& \longrightarrow & M(\beta(d+1))\\
	& (A_{i,j}) & \longmapsto & \left(\frac{d-(i+j-3)}{d+1}A_{i,j}\right)
	\end{array}
	$$
	Then $\alpha_{d}(H_{d})\subseteq H_{d+1}$ and $\alpha_d(Rank_{h}(M(\beta(d))\cap M_{d}))\subseteq Rank_{2h}(M(\beta(d+1)))\cap M_{d+1}$.
\end{Corollary}
\begin{proof}
	Since $\alpha_{d}(A_{i,j})=\alpha_{d}(A_{h,k})$ when $i+j=h+k$ we have that $\alpha_{d}(H_{d})\subseteq H_{d+1}$. By Proposition \ref{p5} $Rank_{h}(M(\beta(d))\cap H_{d} = \sec_{h}(V_{d})$, and by construction $\alpha_{d}(M_{d})$ is the linear change of coordinates mapping a binary form $F\in k[x_0,x_1]_{d}$ to $F'=x_{0}F\in k[x_0,x_1]_{d+1}$.
	
Since $\mathbb{S}ec_{h}(V_{d})\subseteq\mathbb{S}ec_{2h}(V_{d+1})\cap \mathbb{T}^{d}_{[x^{d+1}]}V_{d+1}$, if an $h\times h$ minor of a general matrix $B$ in $M(\beta(d))$ does not vanish, under the assumption that all the $(h+1)\times (h+1)$ minors of $B$ vanish, then there is a $2h\times 2h$ minor of $\alpha_d(B)$ that does not vanish. 
\end{proof}

When $n\geq 2$ we are able to determine, via Macaulay2 \cite{Mc2} aided methods, the rank of $x_0 F$ in some special cases.

\begin{subsubsection}{$(n,d)=(2,2)$}        
The variety  $\mathbb{S}ec_{3}(\mathcal{V}_3^2)$ is the hypersurface
in $\mathbb{P}^9$ cut out by the Aronhold invariant, see for instance
\cite[Section 1.1]{LO15}. With a Macaulay2 computation we prove that if
$F\in\mathbb{S}ec_{2}(\mathcal{V}_2^2)$ is general then the Aronhold
invariant does not vanish at $x_0F$, hence $\rk x_0F=2\rk F$.  
\end{subsubsection}
	
\begin{subsubsection}{$(n,d)=(2,3)$}
The varieties $\mathbb{S}ec_{5}(\mathcal{V}_4^2)$ and
$\mathbb{S}ec_{3}(\mathcal{V}_3^2)$ are both hypersurfaces, given
respectively by the determinant of the catalecticant matrix of second
partial derivatives and the Aronhold invariant
\cite[Section 1.1]{LO15}. With Macaulay2 we prove that the determinant
of the second catalecticant matrix does not vanish at $x_0F$ for $F\in
\mathbb{S}ec_{3}(\mathcal{V}_3^2)$ general, hence $\rk x_0F=2\rk F$.   
\end{subsubsection}

\begin{subsubsection}{$(n,d)=(3,3)$}\label{33}
The secant variety $\mathbb{S}ec_{9}(\mathcal{V}_4^3)$ is the
hypersurface cut out by the second catalecticant matrix
\cite[Section 1.1]{LO15} while $\mathbb{S}ec_{5}(\mathcal{V}_3^3)$
is the entire osculating space. A Macaulay2 computation shows that
$\mathbb{T}_{[x_{0}^{4}]}^{3}\mathcal{V}_4^3\subseteq\mathbb{S}ec_{9}(\mathcal{V}_4^3)$. This
proves that $\rk x_0F<2\rk F$, for $F$ general.
\end{subsubsection}

\begin{subsubsection}{$(n,d)=(4,3)$}\label{34}
In this case
$\mathbb{S}ec_{8}(\mathcal{V}_3^4)=\mathbb{T}_{[x_{0}^{4}]}^{3}\mathcal{V}_4^4$
and $\mathbb{S}ec_{14}(\mathcal{V}_4^4)$ is given by the determinant
of the second catalecticant matrix \cite[Section 1.1]{LO15}. Again
using Macaulay2 we show that
$\mathbb{T}_{[x_{0}^{4}]}^{3}\mathcal{V}_4^4\subseteq
\mathbb{S}ec_{14}(\mathcal{V}_4^4)$.  This
proves that $\rk x_0F<2\rk F$, for $F$ general.
\end{subsubsection}

\begin{Corollary}
For the osculating varieties $T^{3}\mathcal{V}_4^3$ and $T^{3}\mathcal{V}_4^4$ we have 
$$T^{3}\mathcal{V}_4^3\subseteq\mathbb{S}ec_{9}(\mathcal{V}_4^3), \quad T^{3}\mathcal{V}_4^4\subseteq\mathbb{S}ec_{14}(\mathcal{V}_4^4)$$
\end{Corollary}
\begin{proof}
The action of $PGL(n+1)$ on $\mathbb{P}^n$ extends naturally to an action on $\P^{N(n,d)}$ stabilizing $\mathcal{V}_d^n$ and more generally the secant varieties $\sec_h(\mathcal{V}_d^n)$. Since this action is transitive on $\mathcal{V}_d^n$ we have $\mathbb{T}_{[x_{0}^{d}]}^{a}\mathcal{V}_d^n\subseteq\mathbb{S}ec_{h}(\mathcal{V}_d^n)$ if and only if $\mathbb{T}_{[L^{d}]}^{a}\mathcal{V}_d^n\subseteq\sec_h\mathcal{V}_d^n$ for any point $[L^d]\in\mathcal{V}_d^n$ that is $T^a\mathcal{V}_d^n\subseteq\sec_h\mathcal{V}_d^n$. Finally, we conclude by applying \ref{33} and \ref{34}.   
\end{proof}

\stepcounter{thm}
\subsection{Macaulay2 implementation}\label{mac2}
In the Macaulay2 file \texttt{Comon-1.0.m2} we provide a function called \texttt{Comon} which operates as follows:
\begin{itemize}
\item[-] \texttt{Comon} takes in input three natural numbers $n,d,h$;
\item[-] if $h < \binom{n+\lfloor\frac{d}{2}\rfloor}{n}$ then the function returns that Comon's conjecture holds for the general degree $d$ polynomial in $n+1$ variables of rank $h$ by the usual flattenings method in Proposition \ref{p3}. If not, and $d$ is even then it returns that the method does not apply;
\item[-] if $d$ is odd and $\binom{n+k}{n} < 2\binom{n+k-1}{n}$, where $k = \lfloor\frac{d+1}{2}\rfloor$, then again it returns that the method does not apply;
\item[-] if $d$ is odd, $\binom{n+k}{n} \geq 2\binom{n+k-1}{n}$ and $2h-1 > \binom{n+k}{n}$ then it returns that the method does not apply since $2h-2$ must be smaller than the number of order $k$ partial derivatives;
\item[-] if $d$ is odd, $\binom{n+k}{n} \geq 2\binom{n+k-1}{n}$ and $2h-1 \leq \binom{n+k}{n}$ then \texttt{Comon}, in the spirit of Remark \ref{remfun}, produces a polynomial of the form 
$$F = \sum_{i=1}^h(a_{i,0}x_0+ \dots +a_{i,n}x_n)^d$$
then substitutes random rational values to the $a_{i,j}$, computes the polynomial $G = x_0F$, the catalecticant matrix $D$ of order $k$ partial derivatives of $G$, extracts the most up left $2h-1\times 2h-1$ minor $P$ of $D$, and compute the determinant $\det(P)$ of $P$;
\item[-] if $\det(P) = 0$ then \texttt{Comon} returns that the method does not apply, otherwise it returns that Comon's conjecture holds for the general degree $d$ polynomial in $n+1$ variables of rank $h$.
\end{itemize}
Note that since the function \texttt{random} is involved \texttt{Comon} may return that the method does not apply even though it does. Clearly, this event is extremely unlikely. Thanks to this function we are able to prove that Comon's conjecture holds in some new cases that are not covered by Proposition \ref{p3}. Since the case $n = 1$ is covered by Proposition \ref{p5} in the following we assume that $n\geq 2$.

\begin{thm}
Assume $n\geq 2$ and set $h = \binom{n+\lfloor\frac{d}{2}\rfloor}{n}$. Then Comon's conjecture holds for the general degree $d$ homogeneous polynomial in $n+1$ variables of rank $h$ in the following cases:
\begin{itemize}
\item[-] $d = 3$ and $2\leq n \leq 30$;
\item[-] $d = 5$ and $3\leq n \leq 8$;
\item[-] $d = 7$ and $n = 4$.
\end{itemize}
\end{thm}
\begin{proof}
The proof is based on Macaualy2 computations using the function \texttt{Comon} exactly as shown in Example \ref{exp} below. 
\end{proof}

\begin{Example}\label{exp}
We apply the function \texttt{Comon} in a few interesting cases:
\begin{verbatim}
Macaulay2, version 1.12
with packages: ConwayPolynomials, Elimination, IntegralClosure, InverseSystems,
               LLLBases, PrimaryDecomposition, ReesAlgebra, TangentCone
i1 : loadPackage "Comon-1.0.m2";
i2 : Comon(5,3,4)
Lowest rank for which the usual flattenings method does not work = 6
o2 = Comon's conjecture holds for the general degree 3 homogeneous polynomial 
     in 6 variables of rank 4 by the usual flattenings method
i3 : Comon(5,3,6)
Lowest rank for which the usual flattenings method does not work = 6
o3 = Comon's conjecture holds for the general degree 3 homogeneous polynomial 
     in 6 variables of rank 6
i4 : Comon(5,3,7)
Lowest rank for which the usual flattenings method does not work = 6
o4 = The method does not apply --- The determinant vanishes 
i5 : Comon(5,5,21)
Lowest rank for which the usual flattenings method does not work = 21
o5 = Comon's conjecture holds for the general degree 5 homogeneous polynomial 
     in 6 variables of rank 21
i6 : Comon(4,7,35)
Lowest rank for which the usual flattenings method does not work = 35
o6 = Comon's conjecture holds for the general degree 7 homogeneous polynomial
     in 5 variables of rank 35
\end{verbatim}
\end{Example}

\bibliographystyle{amsalpha}
\bibliography{Biblio}
\end{document}